\documentclass[12pt]{amsart}
\usepackage{amsmath,amsthm,amsfonts,amscd,amsxtra,amsopn,amssymb,verbatim,pdfsync,hyperref}
\usepackage[headings]{fullpage}
\usepackage{enumerate}
\usepackage{array}
\usepackage[mathscr]{eucal}
\usepackage{graphicx}
\usepackage{epstopdf}
\usepackage{caption}
\usepackage{xypic}
\usepackage{tikz}
\usetikzlibrary{arrows,decorations.pathmorphing,backgrounds,fit,positioning,shapes.symbols,chains,snakes}
\usetikzlibrary{shapes,arrows}
\usetikzlibrary{calc}

\newtheorem{prop}{Proposition}[section]

\newtheorem{thm}[prop]{Theorem}
\newtheorem{cor}[prop]{Corollary}
\newtheorem{lem}[prop]{Lemma}

\newtheorem{ddefn}[prop]{Definition}
\newtheorem{eex}[prop]{Example}
\newtheorem{rrem}[prop]{Remark}

\newtheorem{eexercise}[prop]{Exercise}
\newtheorem{con}[prop]{Conjecture}
\newtheorem{hhome}[prop]{Homework}
\newtheorem{nnumber}[prop]{}

\newenvironment{defn}{\begin{ddefn}\rm}{\end{ddefn}}
\newenvironment{ex}{\begin{eex}\rm}{\end{eex}}

\newcommand\p{\partial}
\newcommand\pscm{positive scalar curvature metric}
\newcommand{\nb}{\nobreakdash}
\newcommand\wh{\widehat}
\newcommand\opn{\operatorname}
\renewcommand\index{\opn{index}}

\newcommand\scr{\mathscr}
\newcommand{\into}{\hookrightarrow}

\newcommand\C{\mathbb C}

\renewcommand\P{\mathbb P}

\newcommand\Z{\mathbb Z}

\title{An invariant related to the existence of conformally compact Einstein fillings}

\author{Matthew J. Gursky}
\address{Department of Mathematics
         University of Notre Dame\\
         Notre Dame, IN 46556}
\email{\href{mgursky@nd.edu}{mgursky@nd.edu}}

\author{Qing Han}
\address{Department of Mathematics\\
         University of Notre Dame\\
         Notre Dame, IN 46556}
\email{\href{qhan@nd.edu}{qhan@nd.edu}}
\address{Beijing International Center for Mathematical Research\\
         Peking University\\
         Beijing, 100871, China}
\email{\href{qhan@math.pku.edu.cn}{qhan@math.pku.edu.cn}}

\author{Stephan Stolz}
\address{Department of Mathematics
         University of Notre Dame\\
         Notre Dame, IN 46556}
\email{\href{Stephan.A.Stolz.1@nd.edu}{Stephan.A.Stolz.1@nd.edu}}

\begin{document}

\maketitle



\begin{abstract}
We define an invariant for compact spin manifolds $X$ of dimension $4k$ equipped with a metric $h$ of positive Yamabe invariant on its boundary. The vanishing of this invariant is a necessary condition for the conformal class of $h$ to be the conformal infinity of a conformally compact Einstein metric on $X$.
\end{abstract}

\tableofcontents

\section{Introduction}  \label{intro}

Let $\mathring{X}$ be the interior of a compact $(n+1)$-dimensional manifold $X$ with boundary $\partial X = M$.  A Riemannian metric $g_{+}$ on $\mathring{X}$ is said to be {\em conformally compact} if there is a defining function $\rho : X \rightarrow \mathbb{R}$ with $\rho > 0$ in $\mathring{X}$, $\rho = 0$ on $M$, $d\rho \neq 0$ on $M$, such that $\rho^2 g_{+}$ extends to a metric $g$ on $X$.  We will assume that the compactified metric $g$ is at least $C^2$ up to the boundary.   This compactification defines a conformal class of metrics on the boundary: If $h = g_{| M}$ is the induced metric, then the conformal class $[h]$ is called the {\em conformal infinity} of $(\mathring{X},g_{+})$.

If in addition $g_{+}$ satisfies the Einstein condition with negative Einstein constant, which we normalize so that
\begin{align} \label{PEdef}
Ric_{g_{+}} = - n g_{+},
\end{align}
then $g_{+}$ is called a {\em CCE} (for `conformally compact Einstein') or {\em Poincar\'e-Einstein} metric, and $(\mathring{X}, g_{+})$ is a CCE manifold.  The basic example of a CCE manifold is the Poincar\'e model of hyperbolic space with $X = D^{n+1}$, the closed unit disk in $\mathbb{R}^{n+1}$, and $g_{+} = g_{\mathbb{H}}$ the hyperbolic metric given by
\begin{align*}
g_{\mathbb{H}} = \dfrac{4}{(1-|x|^2)^2}dx^2.
\end{align*}
In this example the conformal infinity is the standard conformal structure on the round sphere $S^n$.  CCE manifolds play a fundamental role in the Fefferman-Graham theory of conformal invariants (see \cite{FG}), and in the AdS/CFT correspondence in quantum field theory (see, for example, \cite{Maldecena2}).

In this paper we are interested in the existence problem for CCE metrics:  given a conformal class of metrics on the boundary $M = \partial X$, does the interior admit a CCE metric whose conformal infinity is the given conformal class (i.e., a CCE `filling')?  In addition to explicit examples there are a number of perturbative constructions in the literature; e.g., \cite{AndersonCFT}, \cite{AndersonGAFA}, \cite{OB1}, \cite{OB2}, \cite{OB3}, \cite{CE}, \cite{CD2}, \cite{CD1}, \cite{GLee1991}, \cite{Lee2006},  \cite{MazzeoPacard}, \cite{MV}).  However, in \cite{GH} the first two authors showed that the existence problem is not always solvable: the (standard) seven-dimensional sphere admits infinitely many metrics, which can be chosen in different components of the space of positive scalar curvature metrics, whose conformal classes cannot be the conformal infinity of a CCE metric defined on the eight-dimensional ball.  The same construction can be used to prove obstructions to existence for conformal classes on all spheres $S^{4k - 1}$, where $k \geq 2$.

In this note we extend the work of \cite{GH} in several significant ways.  To give an overview of our results it will be convenient to introduce the following terminology:

\begin{defn} Let $X$ be a compact $(n+1)$-dimensional manifold with boundary $\partial X = M$, and $h$ a Riemannian metric on $M$.  We say that $(X,h)$ is {\em fillable} if there is a CCE metric $g_{+}$ defined in $\mathring{X}$ whose conformal infinity is given by $[h]$. 
\end{defn}

In Section \ref{alphaDef} we introduce an index invariant that provides an obstruction to the existence of CCE fillings for $(X,h)$, where $h$ is a metric with positive Yamabe invariant on the boundary of a compact spin manifold $X$ of dimension $4k$. This  integer valued invariant $I(X,h)$ is given by the index of the Dirac operator on $X$ with respect to a metric $g$ on $X$ which restricts to $h$ on $\p X$, and for which the boundary is totally geodesic.  A key observation is that the vanishing of $I(X,h)$ is a necessary condition for $(X,h)$ to be fillable.

Using properties of the index we then prove an important ``Gluing Lemma'' for the invariant $I$ (Lemma \ref{lem:gluing}).  As an application of the lemma and the fact that \pscm s can be extended over handles of codimension $\ge 3$,  we construct manifolds with boundary $X$ and Riemannian metrics $h$ defined on $M = \partial X$ for which $I(X,h) \neq 0$.  In particular, we prove

\begin{thm}\label{introThm} Let $X$ be compact spin manifold of dimension $4k\ge 8$ with boundary $M$, and suppose that the space $\mathscr{Y}^{+}(M)$ of metrics on $M$ with positive Yamabe invariant is non-empty.
Then there are infinitely many connected components of $\mathscr{Y}^{+}(M)$ such that for any metric $h$ in the component, $(X,h)$ is not fillable.
\end{thm}

The main result of \cite{GH} is a special case of Theorem \ref{introThm}, where $X= D^8$, the $8$\nb-dimensional disk.
However, this theorem produces many other interesting examples of non-fillable pairs $(X,h)$. For example, the AdS Schwarzschild metrics of Hawking-Page give a family of CCE metrics on $D^2 \times S^{n-1}$ ($n\geq 3$) whose conformal infinities are given by the conformal class of the product metric on $S^1(r_m) \times S^{n-1}$, where $r_m > 0$ depends on the mass (see \cite{HP}).  Since $\mathscr{Y}^{+}(S^1 \times S^{n-1})$ is non-empty, by Theorem \ref{introThm} there are infinitely many connected components of $\mathscr{Y}^{+}(S^1 \times S^{n-1})$ containing no metrics $h$ such that $(D^2 \times S^{n-1}, h)$ is fillable, provided $n=4k-1\ge 7$.   The construction of Hawking-Page also works if $S^{n-1}$ is replaced by an $(n-1)$-dimensional Einstein manifold with Einstein constant normalized to have the same value as the sphere, so we obtain the following corollary:

\begin{cor} \label{CorA}  Let $(N,g_N)$ be a closed, $(n-1)$-dimensional Einstein manifold with $Ric(g_N) = (n-2)g_N$, where $n = 4k-1 \geq 7$.  Then the Hawking-Page construction gives a CCE metric on $D^2\times N$, but there are infinitely many connected components of $\mathscr{Y}^{+}(S^1 \times N)$ containing no metrics $h$ with $(D^2 \times N, h)$ fillable.
\end{cor}

The proof of  Theorem \ref{introThm} is based on modifying a given metric $h\in \scr Y^+(\p X)$, essentially by replacing it by the  connected sum $h\# h'$, where $h'$ is a \pscm\ on the sphere $S^{4k-1}$ with $I(D^{4k},h')\ne 0$ and using the Gluing Lemma to show $I(X,h\# h')=I(X,h)+I(D^{4k},h')$ (see Proposition \ref{prop:pscm_conn_sum}). The dimension restriction $4k\ge 8$ in Theorem \ref{introThm} comes from the fact that for $k=1$ the space of \pscm s on $S^3$ is connected by a result of Marquez \cite{Mar}, and hence $I(D^4,h')=0$ for any metric $h'\in \scr Y^+(S^3)$.   

Although we cannot produce $4$\nb-dimensional examples with $I(X,h)\ne 0$ by modifying the metric $h$, such examples can be constructed by modifying the  topology of the interior of $X$ by replacing $X$ by  the connected sum $X'=X\# Y$ of $X$ with a suitable closed spin manifold $Y$.  Again, the Gluing Lemma is used to calculate $I(X',h)$ and to show that $(X',h)$ is non-fillable in many cases.

\begin{thm} \label{ThmB}  Let $X$ be compact spin manifold of dimension $4k\ge 4$ with boundary $M$, and let $h$ be a Riemannian metric on $M$ with positive Yamabe invariant.  Then there is a compact spin manifold $X'$ with $\partial X' = M$, such that that $I(X',h) \neq 0$.  In particular, $(X',h)$ is not fillable.  \end{thm}

In Section \ref{sec:exotic_pscm} we prove a result (Corollary \ref{cor:Ahat}) that is a slightly stronger statement than Theorem \ref{ThmB}  that answers an interesting question posed to us by S. Alexakis: is there a closed Riemannian manifold $(M,h)$ such that for every manifold $X$ with $\partial X = M$, we have $I(X,h) \neq 0$?  The answer turns out to be no:

\begin{prop}  \label{prop:intro} Suppose $M$ is the boundary of a compact spin manifold $X$ of dimension $4k \geq 4$, and let $h$ be a Riemannian metric of positive Yamabe invariant on $M$.  Then there is a compact spin manifold $X'$ with $\p X'=M$ such that $I(X',h) = 0$.
\end{prop}

One can broaden the existence question for CCE metrics and ask for a given manifold $M$ and metric $h\in \scr Y^+(M)$, whether we can always find a manifold $X$ with $\partial X = M$ such that $(X,h)$ is fillable.  (Of course, one should assume that $M$ can be realized as the boundary of some manifold).  In view of the above proposition, the invariant $I$ cannot be used to detect a counterexample to this version of the existence problem.

We conclude with a remark about the organization of the paper.  In Section \ref{alphaDef} we define the invariant $I$, and establish some of its basic properties.  In Section \ref{sec:ex} we state the main technical results that are used in the proof of Theorems \ref{introThm} and \ref{ThmB}.  These results are based on the Gluing Lemma for the invariant $I$, and the construction of positive scalar curvature metrics on $S^{4k-1}$ with non-trivial invariant when $k \geq 2$.  The former is proved in Section \ref{sec:gluing}, while the latter is carried out in Section \ref{sec:exotic_pscm}.     \\

The first author acknowledges the support of NSF grants DMS-1509633 and DMS-1547292.  The second author acknowledges the support of NSF grant DMS-1404596. The third author acknowledges the support of NSF grant DMS-1547292.

\section{Defining the invariant} \label{alphaDef}

To define the invariant $I$ we need to recall some basic facts about the Dirac operator on spin manifolds with boundary (see, for example, \cite{BBW}).

Let $(X, g)$ be a smooth compact Riemannian spin manifold of dimension $4k$ with boundary $M = \partial X$.  Let $S(X)$ denote the associated spinor bundle, and $D = D(X,g) : S(X) \rightarrow S(X)$ the Dirac operator with respect to the Atiyah-Patodi-Singer boundary conditions.  Since the dimension is even,   $S(X)$ splits into the sub-bundles of positive and negative spinors: $S(X) = S^{+}(X) \oplus S^{-}(X)$.  This induces a splitting of the Dirac operator $D = D(X,g)$, given by
\begin{align*}
D = \left( \begin{array}{cc}
0 & D^{-} \\
D^{+} & 0 \end{array} \right).
\end{align*}
Moreover, the operator $D^{+} = D^{+}(X,g)$ is Fredholm with $L^2$-formal adjoint given by $D^{-}$.  The index  of $D^+(X,g)$ is well defined (see Chapters 7 and 20 of \cite{BBW}), and is given by
\begin{align} \label{idef}
\opn{index}(D^+(X,g)) = \dim\mathscr H^{+}(X,g) - \dim\mathscr H^{-}(X,g),
\end{align}
where $\mathscr H^{\pm}( X,g)$ is the space of positive/negative harmonic spinors.   Since the space of harmonic spinors is conformally invariant (see \cite{Hitchin}, Proposition 1.3), the index is also conformally invariant.

Assuming the metric is a product near the boundary, the classic result of Atiyah-Patodi-Singer (\cite{APS}, Theorem 4.2) gives an explicit formula for the index.   In Remark (2) on page 61 of \cite{APS}, the authors comment that it is also possible to give an expression for general metrics, but the formula needs to be modified by adding a term involving the second fundamental form of the boundary.  A discussion of the correction term appears in (\cite{EGH}, pp. 348 -- 349).  For the signature complex explicit formulas appear in \cite{GilkeyPP} and (\cite{DW}, Section 2), and these formulas have obvious extensions to the case of the Dirac operator.    The upshot is that when the boundary is totally geodesic, the correction term vanishes and we have
\begin{equation}\label{eq:APS1}
\opn{index}(D^+(X,g))
=\int_{X}\wh A(p) - \frac12\eta(D(\p X,g_{|\p X})) - \frac 12\dim\mathscr H(\p X, g_{| \p X}),
\end{equation}
where
\begin{itemize}
\item $\wh A(p)$ is the Hirzebruch $\wh A$\nb-polynomial applied to the Pontryagin forms $p_i$ determined by the Riemannian metric on $X$; \vskip.05in
\item $\eta(D(\p X,g_{|\p X}))$ is the value of the $\eta$\nb-function of the Dirac operator on $\p X$ at $s=0$; \vskip.05in
\item $\mathscr H(\p X, g_{|\p X})$ is the space of harmonic spinors on $\partial X = M$. \vskip.1in
\end{itemize}

Now suppose we are given a metric $h$ on $M = \partial X$.  To define the index we need to extend $h$ to $X$, and in view of the preceding comments we will restrict to {\em totally geodesic extensions}; i.e., metrics $g$ in $X$ with $g_{| M} = h$ such that $M$ is totally geodesic with respect to $g$.

\begin{lem}\label{lem:index} Assume $h$ is a metric on $M = \partial X$ with positive Yamabe invariant $Y(M,[h]) > 0$.  If $g$ is a totally geodesic extension of $h$, then the index is given by
\begin{equation}\label{eq:APS2}
\opn{index}(D^+(X,g))
=\int_{X}\wh A(p) - \frac12\eta(D(\p X,h)).
\end{equation}
\end{lem}

\begin{proof}   Since $h$ has positive Yamabe invariant and $\mathscr H(\p X, g_{ \p X})$ is conformally invariant, it follows from the vanishing result of Lichnerowicz \cite{Lich} that $\dim\mathscr H(\p X, h) = 0$.  Therefore, (\ref{eq:APS2}) follows from (\ref{eq:APS1}).  \end{proof}

Our next observation is that index does not depend on the choice of the totally geodesic extension:

\begin{prop}\label{prop:1} Let $X$ be a smooth, compact spin manifold of dimension $4k$ with boundary, and let $h$ be a metric on $M = \partial X$ with positive Yamabe invariant.  If $g_0$ and $g_1$ are totally geodesic extensions of $h$, then $\opn{index}(D^+(X,g_0)) =  \opn{index}(D^+(X,g_1))$.

More generally, the index only depends on $X$ and the path component of $h$ in $\mathscr{Y}^{+}(M)$, the space of Riemannian metrics with positive Yamabe invariant on $M$.
\end{prop}

\begin{proof} Using the formula \eqref{eq:APS2}, it is easy to see that the index does not depend on the choice of extension.  However, to prove the more general claim that the index only depends on the path component of $h$ in $\mathscr{Y}^{+}(M)$ we will give an argument that uses basic properties of the index, and does not rely on an explicit formula.

 First, suppose $g_0$ and $g_1$ are totally geodesic extensions, and let $\{ g(t)\}_{0 \leq t \leq 1}$ be a continuous path of totally geodesic extensions of $h$ with $g(0) = g_0$ and $g(1) = g_1$.  (We note that such a path can be constructed in various ways; for example, we can first deform $g_0$ and $g_1$ to be product metrics near the boundary, and then it is easy to connect two such metrics by a continuous path of totally geodesic extensions).   Since the induced metric on the boundary is fixed, the term in the index formula (\ref{eq:APS2}) given by the $\eta$-invariant does not depend on $t$.  Also, the Pontryagin forms depend continuously on the metric, so the integral of the Hirzebruch $\wh A$\nb-polynomial varies continuously with $t$.  It follows that the index varies continuously as well, but since it is integer-valued it must be constant.

More generally, suppose $h_0$ and $h_1$ are two metrics in the same path-connected component of $\mathscr{Y}^{+}(M)$.  Let $\{ h(t) \}_{0 \leq t \leq 1}$ be a continuous path in this component with $h(0) = h_0$ and $h(1) = h_1$.  Also, let $\{ g(t)\}_{0 \leq t \leq 1}$ be a continuous path of metrics in $X$ such that each $g(t)$ is a totally geodesic extension of $h(t)$.  If $\widetilde{g}(t)$ is another path of totally geodesic extensions with $\widetilde{g}(t)_{|M} = h(t)$, then by the preceding argument
\begin{align*}
\opn{index}(D^+(X,\widetilde{g}(t))) =  \opn{index}(D^+(X,g(t)))
\end{align*}
for each $t$. Therefore, the index is independent of the path of totally geodesic extensions we choose.  Arguing as before, we want to show that the index does not depend on $t$, although in this case the boundary metric is not fixed and we need to use the invariance of the index under continuous deformations of the operator.   In particular, we need to show that the APS boundary condition (see \cite{APS}, Section 2) depend continuously on the parameter $t$.

Recall that $D(M,h(t))$, the Dirac operator on $M$, has discrete spectrum with eigenvalues
\begin{align*}
\cdots \leq \lambda_{-2}(t) \leq \lambda_{-1}(t) < 0 \leq \lambda_0(t) \leq \lambda_1(t) \leq \lambda_2(t) \leq \cdots
\end{align*}
If $\mathcal{S}$ denotes the spinor bundle of $M$, we can choose an orthonormal basis of $L^2(M,\mathcal{S})$ consisting of eigenspinors $\psi_i$ with $D(M,h(t))\psi_i = \lambda_i(t) \psi_i$.  The APS boundary condition is a non-local condition defined in terms of the projection operator
\begin{align*}
P_{\geq 0}^t : L^2(M,\mathcal{S}) \rightarrow L^2(M,\mathcal{S}),
\end{align*}
where $P_{\geq 0}^t$ is the projection onto the space spanned by the eigenspinors $\psi_i$ with $\lambda_i(t) \geq 0$.
Since $h(t)$ is in $\mathscr{Y}^{+}(M)$, there is a positive lower bound for the first eigenvalue $\mu_1(M,h(t))$ of the conformal laplacian:
\begin{align} \label{gap}
\mu_1(M,h(t)) \geq \mu_0 > 0, \  0 \leq t \leq 1.
\end{align}
It follows from (\cite{Hijazi}, Theorem A) that
\begin{align*}
\lambda_0(t)^2 \geq \dfrac{4k-1}{8(2k-1)}\mu_0 > 0.
\end{align*}
Therefore, the positive spectrum has a uniform lower bound for all $0 \leq t \leq 1$, and $P_{\geq 0}^t$ varies continuously with $t$.  As before we conclude that the index is constant in $t$.    \end{proof}

In view of the preceding result it is natural to make the following definition:

\begin{defn} Let $X$ be a smooth compact spin manifold of dimension $4k$ with boundary $M$, and let $h$ be a metric on $M$ with positive Yamabe invariant.  We define the invariant $I(X,h)\in \Z$ by
\begin{align} \label{alphadef}
I(X,h):= \opn{index}(D^+(X,g))
\end{align}
where $g$ is any totally geodesic extension of $h$ (by Proposition \ref{prop:1}, $I(M,h)$ is independent of the choice of extension $g$).  If $X$ is closed (i.e., has no boundary), then we write $I(X)$ instead of $I(X,h)$. By the Index Theorem  $I(X)$ is equal to $\wh A(X):=\int_X \wh A(p)$, Hirzebruch's $\wh A$\nb-genus.
\end{defn}

By Proposition \ref{prop:1}, $I$ is also a conformal invariant of the boundary metric.  In fact, the proposition tells us that it can be viewed as a mapping $I : \mathscr{Y}^{+}(M) \rightarrow \Z$ that is constant on each connected component of the domain.  However, to keep the notation simple we will use $I(M,h)$ to denote the invariant, despite its somewhat inaccurate connotations.

By appealing to a result of Hijazi-Montiel-Rold\`{a}n \cite{HMR}, we can give conditions which imply the vanishing of the $I$-invariant:

\begin{lem}{\bf (Vanishing Lemma)} \label{lem:vanishing} Let $X$ be a smooth compact spin manifold of dimension $4k$ with boundary $M$, and let $h$ be a Riemannian metric on $M$ with positive Yamabe invariant.  If the conformal class of $h$ contains a representative $h_0$ that admits a totally geodesic extension $g$ of positive scalar curvature, then $I(X,h) = 0$.
\end{lem}

\begin{proof} This is an immediate corollary of the following:

\begin{thm} (See Theorem 2 of \cite{HMR})  Let $(X,g)$ be a compact Riemannian spin manifold of dimension $n$ with non-empty boundary $M$, and assume $M$ has
non-negative mean curvature (with respect to the inner normal). Under the Atiyah-Patodi-Singer boundary
condition, the spectrum of the Dirac operator $D(X,g)$ is a sequence of unbounded real numbers $\{ \lambda^{APS}_k \vert k \in \mathbb{Z} \}$ satisfying
\begin{align} \label{gapk}
\big( \lambda_k^{APS} \big)^2 \geq \dfrac{n}{4(n-1)} \inf_{X} R_g,
\end{align}
for all $k \in \mathbb{Z}$, where $R_g$ is the scalar curvature of $g$.
\end{thm}

If $[h]$ is a conformal class on $M =\partial X$ containing a metric $h_0$ that admits a totally geodesic extension $g$ with scalar curvture $R_g > 0$, then $M$ has zero mean curvature and (\ref{gapk}) holds.  It follows that there are no harmonic spinors, hence $\opn{index}(D^+(X,g)) = I(X,h) = 0$.
\end{proof}

Using the preceding result, we immediately deduce the following obstruction to CCE fillings:

\begin{cor} Let $X$ be a smooth compact Riemannian spin manifold of dimension $4k$ with boundary $M = \partial X$.  If $h$ is a Riemannian metric on $M$ with positive Yamabe invariant  such that $I(X,h)\neq 0$, then $(X,h)$ is not fillable.
\end{cor}

\begin{proof} Suppose $X$ admits a CCE metric $g_{+}$ with conformal infinity given by $(M,[h])$. By \cite{CQY} (based on \cite{Lee}), there is a compactification $g = \rho^2 g_{+}$ with positive scalar curvature such that the boundary $M$ is totally geodesic.  By \cite{CDLS}, we may assume that $g$ is smooth up to the boundary. If we let $h_0 = g_{| M}$ denote the induced metric on the boundary, then $g$ is an extension of $h_0 \in [h]$ with positive scalar curvature and zero mean curvature, and it follows from the Vanishing Lemma \ref{lem:vanishing} that $I(X,h) = 0$.
\end{proof}

\section{Examples of manifolds $X$ with $I(X,h)\ne 0$}\label{sec:ex}

In this section we state results concerning the $I$-invariant and show how these imply Theorem \ref{introThm}. These results will be proved in the remaining sections of the paper.

A first step is to construct \pscm s $h$ on the sphere $S^{4k-1}$ for $k\ge 2$ with  $I(D^{4k},h)\ne 0$. Gromov and Lawson proved in \cite[Theorem 4.47]{GL} that the space $\scr R^+(S^7)$ of \pscm s on $S^7$ has infinitely many connected components, and the main result of \cite{GH} shows that most of these metrics are not conformal infinities of CCE metrics on the $8$\nb-dimensional disk $D^8$. More generally, according to Theorem 2.6 of \cite{Rosenberg}, there are infinitely many connected components of $\scr R^+(S^{4k-1})$ for $k\ge 2$. The invariant used to distinguish these components is an index invariant, which (up to a factor of $2$ for $k$ odd) can be identified with the invariant $I(D^{4k},h)$. In particular, this result shows that for fixed $k\ge 2$ the integers $I(D^{4k},h)$ as $h$ ranges over the \pscm s on $S^{4k-1}$ form an {\em infinite} subset (in fact, an infinite subgroup) of $\Z$. This implies that there are infinitely many connected components of the space $\scr R^+(S^{4k-1})$, $k\ge 2$, which cannot be conformal infinities of CCE metrics on the disk $D^{4k}$. In particular, this proves Theorem \ref{introThm} for disks of dimension $4k\ge 8$.

In this paper we obtain the following precise characterization of the possible values of $I(D^{4k},h)$ for $h\in \scr R^+(S^{4k-1})$.

\begin{prop}\label{prop:exotic_pscm}
\[
\left\{\left.I(D^{4k},h)\in \Z\right| h\in \scr R^+(S^{4k-1})\right\}=
\begin{cases}
0&k=1\\
\Z&\text{$k\ge 2$, $k$ even}\\
2\Z&\text{$k\ge 2$, $k$ odd}
\end{cases}
\]
\end{prop}

It is known that the space $\scr R^+(S^3)$ of \pscm s on $S^3$ is path connected \cite{Mar}. It follows that $I(D^4,h)=I(D^4,h_0)$, where $h_0$ is the round metric on $S^3$. Moreover, $I(D^4,h_0)=0$ by Lemma \ref{lem:vanishing}, since $h_0$ extends to a totally geodesic \pscm\ on $D^4$. This proves the proposition for $k=1$. The proof for $k\ge 2$ is the subject of section \ref{sec:exotic_pscm}.

To prove Theorem \ref{introThm} for general compact spin manifolds $X$ of dimension $4k\ge 8$, we need to construct infinitely many \pscm s $h$ on $M=\p X$ in different connected components of the space $\scr R^+(M)$ of \pscm s on $M$, and show that $I(X,h)\ne 0$. To do this, we use the connected sum construction for \pscm s.
Let $M_1$, $M_2$ be closed manifolds of dimension $n$, and let $M_1\# M_2$ be their connected sum (given by removing an open $n$\nb-disk from both manifolds, and gluing the resulting manifolds  along their boundary $S^{n-1}$).
If $h_1$, $h_2$ are \pscm s on $M_1$ resp.\ $M_2$, then a \pscm\ $h_1\# h_2$ can be constructed on $M_1\# M_2$, provided $n\ge 3$
(see Definition \ref{def:conn_sum_pscm}).

If $M_2$ is a sphere, then the connected sum $M_1\# M_2$ is diffeomorphic to $M_1$, and hence $h_1\# h_2$ can be interpreted as a new \pscm\ on $M_1$.  Using this observation, we will prove

\begin{prop}\label{prop:pscm_conn_sum} Let $h_1$ be a \pscm\ on the boundary $M=\p X$ of a compact spin manifold $X$ of dimension $4k\ge 4$, and let $h_2\in \scr R^+(S^{4k-1})$. Then
\[
I(X,h_1\# h_2) =I(X,h_1)+I(D^{4k},h_2).
\]
\end{prop}

In particular, the previous Proposition then implies that for $4k\ge 8$ there are infinitely many components of $\scr R^+(M)$ such that for any metric $h$ in that component $I(X,h)\ne 0$, thus proving Theorem \ref{introThm} of the introduction. The above proposition will be proved in section \ref{sec:gluing}.

\smallskip

If $X$ is a $4$\nb-dimensional compact spin manifold with boundary $M$,  $h_1\in \scr R^+(M)$, and $h_2\in \scr R^+(S^3)$, then according to the two propositions above,
\[
I(X,h_1\# h_2)=I(X,h_1)+I(D^4,h_2)=I(X,h_1).
\]
In other words, unlike for $k\ge 2$, we cannot produce examples of \pscm s $h$ with $I(X,h)\ne 0$ by replacing $h$ by $h\# h_2$ for suitable $h_2\in \scr R^+(S^3)$.  Instead, to obtain $4$\nb-dimensional examples with non-trivial $I$\nb-invariant, we will replace $X$ by the connected sum $X\# Y$ of $X$ with  a closed spin manifold $Y$.

\begin{prop}\label{prop:conn_sum}
Let $h$ be a \pscm\ on the boundary $M=\p X$ of a compact spin manifold $X$ of dimension $4k \geq 4$, and let $Y$ be a closed spin manifold of dimension $4k$. Then
\[
I(X\#Y,h)=I(X,h)+\wh A(Y),
\]
where $X\# Y$ is the connected sum of $X$ and $Y$ (this is again a manifold with boundary $M$).
\end{prop}

This proposition will be proved in section \ref{sec:gluing}.  Note that Theorem \ref{ThmB} is an immediate consequence of this formula.

\smallskip

Using the preceding results we can give examples of spin $4$-manifolds with $I(X,h)\ne 0$.

\begin{ex} If $h$ is a metric with positive Yamabe invariant on $S^3 = \partial D^4$, then for any closed spin $4$\nb-manifold $Y$
\[
I(D^4\# Y,h)=I(D^4,h)+\wh A(Y)=\wh A(Y).
\]
This proves that for any $Y$ with $\wh A(Y)\ne 0$ (for example the Kummer surface, see Equation \eqref{eq:Kummer}), $(D^4\# Y, h) = (Y\setminus\mathring D^{4k},h)$ is not fillable.
\end{ex}

\begin{ex} The same argument applies to \pscm s $h$ on other $3$\nb-manifolds. For example, let $M=S^1\times S^2$, and let $h$ be the standard product metric on $M$. Then $I(D^2\times S^2,h)=0$, since $h$ extends to the  product metric $g$ on $D^2\times S^2$ (which is the hemisphere metric on the factor $D^2$ and the round metric on the factor $S^2$) which has positive scalar curvature, and for which the boundary is totally geodesic. Hence for a closed spin manifold $Y$ of dimension $4$, $I((D^2\times S^2)\# Y,h)=\wh A(Y)$. This shows that
$((D^2\times S^2)\# Y,h)$ is not fillable, for any $Y$ with $\wh A(Y)\ne 0$.  However, as explained in the Introduction, the AdS Schwarzschild example of Hawking-Page is a Poincar\'e-Einstein metric on $D^2 \times S^2$ with conformal infinity given by a product metric on $S^1 \times S^2$.
\end{ex}

The rest of the paper is devoted to the proof of the three propositions above. In the next section we prove a gluing lemma for the $I$\nb-invariant and show that it implies Propositions \ref{prop:pscm_conn_sum} and \ref{prop:conn_sum}.

 \section{A gluing formula for $I(X,h)$}  \label{sec:gluing}

Let $X$ be a compact spin manifold of dimension $4k$  with boundary $\p X=M_-\amalg M_+$, i.e., the boundary of $X$ is the disjoint union of closed manifolds $M_-$ and $M_+$. The manifolds $M_\pm$ are not required to be connected, and they are allowed to be empty. Suppose that $X$ decomposes in the form $X=X_-\cup_M X_+$, where $M$ is a hypersurface of $X$, as indicated in the picture below.

\medskip
\begin{center}
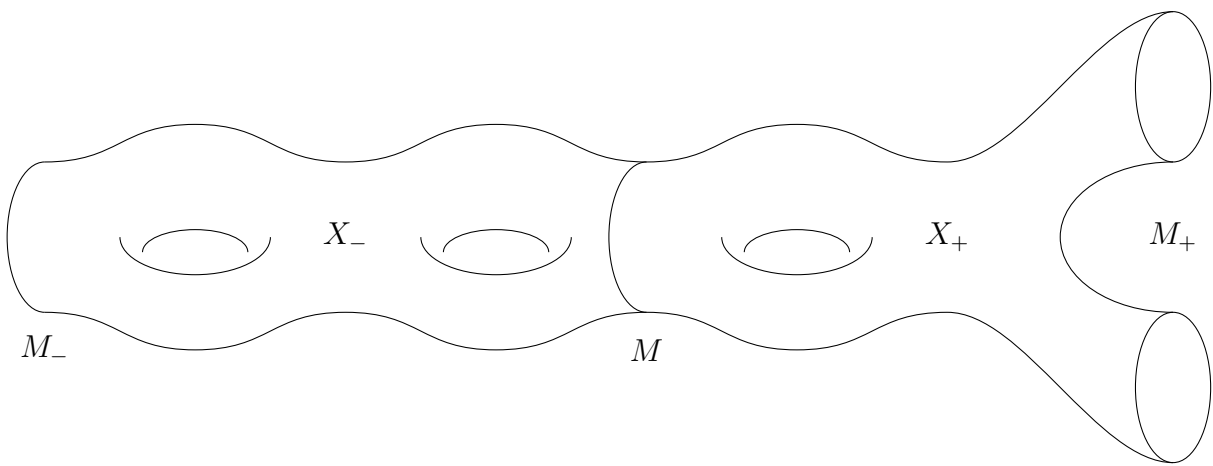

\begin{tikzpicture}
\draw (0,1) arc (90:270:.5 and 1);
\draw (0,1).. controls (1,1) and (1,1.5) ..
(2,1.5) .. controls (3,1.5) and (3,1) ..
(4,1).. controls (5,1) and (5,1.5) ..
(6,1.5) .. controls (7,1.5) and (7,1) ..
(8,1).. controls (9,1) and (9,1.5) ..
(10,1.5).. controls (11,1.5) and (11,1) ..
(12,1).. controls (13,1) and (14,3) ..
(15,3);
\draw (0,-1) .. controls (1,-1) and (1,-1.5) ..
(2,-1.5) .. controls (3,-1.5) and (3,-1) ..
(4,-1).. controls (5,-1) and (5,-1.5) ..
(6,-1.5) .. controls (7,-1.5) and (7,-1) ..
(8,-1).. controls (9,-1) and (9,-1.5) ..
(10,-1.5).. controls (11,-1.5) and (11,-1) ..
(12,-1).. controls (13,-1) and (14,-3) ..
(15,-3);
\draw (15,1) arc (90:270:1.5 and 1);
\draw (8,1) arc (90:270:.5 and 1);
\draw (15,-2) circle (.5 and 1);
\draw (15,2) circle (.5 and 1);
\draw (1,0) arc (180:360:1 and .5);
\draw (1.3,-.2) arc (180:0:.7 and .3);
\draw (5,0) arc (180:360:1 and .5);
\draw (5.3,-.2) arc (180:0:.7 and .3);
\draw (9,0) arc (180:360:1 and .5);
\draw (9.3,-.2) arc (180:0:.7 and .3);
\node at (4,0){$X_-$};
\node at (12,0){$X_+$};
\node at (8,-1.5){$M$};
\node at (0,-1.5){$M_-$};
\node at (15,0){$M_+$};
\end{tikzpicture}
\captionof{figure}{The decomposed manifold $X=X_-\cup_M X_+$}
\end{center}

\begin{lem}\label{lem:gluing}{\bf (Gluing Lemma)}
Let $X$ be a compact spin manifold of dimension $4k\ge 4$ with boundary $\p X=M_-\amalg M_+$ with a decomposition $X=X_-\cup_M X_+$ as described above. Let $h_-,h,h_+$ be \pscm s on $M_-$ resp.\ $M$ resp.\ $M_+$. Then
\[
I(X,h_-\amalg h_+)=I(X_-,h_-\amalg h)+I(X_+,h\amalg h_+).
\]
In particular, if $\p X=\emptyset$, then
\[
\wh A(X)=I(M_-,h)+I(M_+,h).
\]
\end{lem}

\begin{proof}
To prove this statement we will compute the index invariant $I$ for the three manifolds $X$, $X_+$ and $X_-$ using Lemma \ref{lem:index} for these manifolds with boundary. This requires that the metric is a product metric near the boundary. So in the first step we pick a Riemannian metric $g$ on $X$ which
\begin{itemize}
\item restricts to $h$ on $M\subset X$ and to $h_\pm$ on $M_\pm$, and
\item is a product metric near these codimension $1$ submanifolds. In other words, a neighborhood $U$ of $X$ is isometric to $(M,h)\times (-1,+1)$ in such a way that $U\cap X_-=M\times (-1,0]$, and $U\cap X_+=M\times [0,+1)$, and a neighborhood of $M_\pm$ is isometric to $(M_\pm ,h_\pm)\times [0,1)$.
\end{itemize}
Then by Lemma \ref{lem:index}
\begin{equation}\label{eq:APS}
\opn{index}(D^+(X,g))
=\int_{X}\wh A(p) - \frac12\eta(D(\p X,g_{|\p X})).
\end{equation}

The boundary $\p X$ is the disjoint union of $M_-$ and $M_+$, but care is needed with orientations, since changing the orientation of a manifold $M$ multiplies the eta invariant of its Dirac operator by $-1$ (as discussed in \cite{APS} just before Theorem 4.2). In particular, in  formula \eqref{eq:APS} the orientation of the boundary $\p X$ is the orientation  {\em induced by the orientation on $X$} and its $\eta$\nb-invariant $\eta(D(\p X,g_{|\p X}),0)$ has to be understood accordingly. This is not so relevant for $M_\pm$, but becomes crucial when we consider $M$, since the orientation of $M$ considered as boundary of $X_+$ is the {\em opposite} of the orientation it has as boundary of $X_-$.

Let us adopt the following orientation conditions on $M, M_\pm$: the orientation of $M_+$ (resp.\ $M$) is induced by the orientation of $X$ (resp.\ $X_-$), while the orientation of $M_-$ is the {\em opposite} of the orientation induced by $X$. In other words, as oriented manifolds, $\p X$ is the disjoint union of $M_+$ and $\bar M_-$, the notation for the  manifold $M_-$ equipped with the opposite orientation. We have
\[
\p X=\bar M_-\amalg M_+
\qquad
\p X_-=\bar M_-\amalg M
\qquad
\p X_+=\bar M\amalg M_+.
\]
This implies
\begin{align*}
\eta(D(\p X,g_{\p X}))&=\eta(D(\bar M_-,h_-))+\eta(D(M_+,h_+))\\
&=-\eta(D(M_-,h_-))+\eta(D(M_+,h_+)).
\end{align*}
It follows that
\begin{align*}
I(X,h_-\amalg h_+)=
&\index(D^+(X,g))\\
=&\int_{X}\wh A(p) - \frac12\eta(D(\p X,g_{|\p X}))\\
=&\int_{X}\wh A(p) - \frac12(-\eta(D(M_-,h_-))+\eta(D(M_+,h_+))).
\end{align*}
Similarly
\[
 I(X_- ,h_-\amalg h)
= \int_{X_-}\wh A(p) - \frac12(-\eta(D(M_-,h_-))+\eta(D(M,h)))
\]
and
\[
 I(X_+ ,h\amalg h_+)
 =\int_{X_+}\wh A(p) - \frac12(-\eta(D(M,h))+\eta(D(M_+,h_+))).
\]
Adding these two equations, the two integrals combine, and the terms $\eta(D(M,h))$ cancel, leaving us with
\begin{align*}
&I(X_- ,h_-\amalg h)+I(X_+ ,h\amalg h_+)\\
=&\int_X\wh A(p) - \frac12(-\eta(D(M_-,h_-))+\eta(D(M_+,h_+)))\\
=&I(X,h_-\amalg h_+)
\end{align*}
as claimed.
\end{proof}

\begin{proof}[Proof of Proposition \ref{prop:conn_sum}]
The connected sum $X\# Y$ is formed by removing an open disk $\mathring D^{4k}$ from $X$ resp.\ $Y$, resulting in manifolds
\[
X\setminus \mathring D^{4k}\quad
\text{with boundary $\p X\amalg S^{4k-1}$}
\qquad \text{and}\qquad
Y\setminus \mathring D^{4k}
\quad
\text{with boundary $S^{4k-1}$},
\]
and then gluing these along the common boundary component $S^{4k-1}$. Here is a picture.

\medskip

\begin{center}
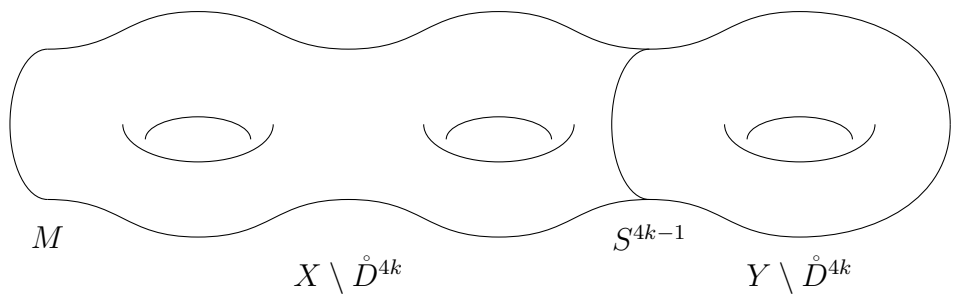

\begin{tikzpicture}
\draw (0,1) arc (90:270:.5 and 1);
\draw (0,1).. controls (1,1) and (1,1.5) ..
(2,1.5) .. controls (3,1.5) and (3,1) ..
(4,1).. controls (5,1) and (5,1.5) ..
(6,1.5) .. controls (7,1.5) and (7,1) ..
(8,1).. controls (9,1) and (9,1.5) ..
(10,1.5).. controls (11,1.5) and (12,1) ..
(12,0);
\draw (0,-1) .. controls (1,-1) and (1,-1.5) ..
(2,-1.5) .. controls (3,-1.5) and (3,-1) ..
(4,-1).. controls (5,-1) and (5,-1.5) ..
(6,-1.5) .. controls (7,-1.5) and (7,-1) ..
(8,-1).. controls (9,-1) and (9,-1.5) ..
(10,-1.5).. controls (11,-1.5) and (12,-1) ..
(12,0);
\draw (8,1) arc (90:270:.5 and 1);
\draw (1,0) arc (180:360:1 and .5);
\draw (1.3,-.2) arc (180:0:.7 and .3);
\draw (5,0) arc (180:360:1 and .5);
\draw (5.3,-.2) arc (180:0:.7 and .3);
\draw (9,0) arc (180:360:1 and .5);
\draw (9.3,-.2) arc (180:0:.7 and .3);
\node at (4,-2){$X\setminus \mathring D^{4k}$};
\node at (10,-2){$Y\setminus \mathring D^{4k}$};
\node at (8,-1.5){$S^{4k-1}$};
\node at (0,-1.5){$M$};
\end{tikzpicture}
\captionof{figure}{The connected sum $X\# Y$}
\end{center}

\medskip

Let $h_0$ be the round metric on $S^{4k-1}$. Then applying Gluing Lemma \ref{lem:gluing} to the manifold $X\# Y$ we obtain
\begin{equation}\label{eq:conn_sum}
I(X\# Y,h)=I(X\setminus \mathring D^{4k},h\amalg h_0)
+I(Y\setminus \mathring D^{4k},h_0).
\end{equation}
Next, applying the Gluing Lemma to the decomposed manifold $X=(X\setminus \mathring D^{4k})\cup_{S^{4k-1}}D^{4k}$ gives
\[
I(X,h)=I(X\setminus \mathring D^{4k},h\amalg h_0)+I(D^{4k},h_0).
\]
The round metric on $D^{4k}$, viewed as a hemisphere of $S^{4k}$, then is a \pscm\ which restricts to $h_0$ on the boundary. This implies $I(D^{4k},h_0)=0$ by Proposition \ref{lem:vanishing}, and hence $I(X\setminus \mathring D^{4k},h\amalg h_0)=I(X,h)$. Similarly, we see that $
I(Y\setminus \mathring D^{4k},h_0)=\wh A(Y)$. Combining this with equation \eqref{eq:conn_sum}, we conclude $I(X\# Y,h)=I(X,h)+\wh A(Y)$ as claimed.
\end{proof}

Proposition \ref{prop:pscm_conn_sum} is a consequence of a more general statement that we prove next.
To state it, we need the following construction.

\begin{defn}\label{def:b_conn_sum}
Let $X_1$, $X_2$ be manifolds of dimension $n+1$ with non-empty boundary $M_1=\p X_1$, $M_2=\p X_2$. The {\em boundary connected sum} $X_1\#_{\p} X_2$ is the manifold of dimension $n+1$ constructed as follows.  Let $f_i\colon D^n\into M_i$ be embeddings of the closed $n$\nb-disk into $M_i$. Then  $X_1\#_{\p} X_2$ is the quotient space of the disjoint union of $X_1$, $X_2$ and the product $D^1\times D^n$ obtained by identifying $(-1,x)\in D^1\times D^n$ with $f_1(x)\in \p X_1$ and  $(1,x)\in D^1\times D^n$ with $f_2(x)\in \p X_2$.
This gluing results in corner points along $\p D^1\times \p D^n$ which can be smoothed to give the boundary connected sum $X_1\#_{\p} X_2$ the structure of a smooth manifold with boundary. We note that the boundary of $X_1\#_{\p} X_2$ is diffeomorphic to the connected sum  $M_1\# M_2$.
\end{defn}

The boundary connected sum of two manifolds of  dimension $2$ is illustrated in figure \ref{fig:b_conn_sum}.

\medskip

\begin{center}
\begin{tikzpicture}[scale=.8]
\draw (0,0).. controls (0,1) and (1,1.5) ..
(2,1.5) .. controls (3,1.5) and (3,1) ..
(4,1).. controls (5,1) and (5,1.5) ..
(6,1.5) .. controls (7,1.5) and (7,1) ..
(8,1)--(9,1);
\draw (0,0).. controls (0,-1) and (1,-1.5) ..
(2,-1.5) .. controls (3,-1.5) and (3,-1) ..
(4,-1).. controls (5,-1) and (5,-1.5) ..
(6,-1.5) .. controls (7,-1.5) and (7,-1) ..
(8,-1)--(9,-1);
\draw (1,0) arc (180:360:1 and .5);
\draw (1.3,-.2) arc (180:0:.7 and .3);
\draw (5,0) arc (180:360:1 and .5);
\draw (5.3,-.2) arc (180:0:.7 and .3);
\draw (9,0) circle (.5 and 1);
\node at (6,.8){$X_1$};
\begin{scope}[shift={(4,-5)}]
\draw (0,0).. controls (0,1) and (1,1.5) ..
(2,1.5) .. controls (3,1.5) and (3,1) ..
(4,1)--(5,1);
\draw (0,0).. controls (0,-1) and (1,-1.5) ..
(2,-1.5) .. controls (3,-1.5) and (3,-1) ..
(4,-1)--(5,-1);
\draw (1,0) arc (180:360:1 and .5);
\draw (1.3,-.2) arc (180:0:.7 and .3);
\draw (5,0) circle (.5 and 1);
\node at (2,.8){$X_2$};
\end{scope}
\draw ($(9,0)+(30:.5 and 1)$) arc (90:-90:3);
\draw ($(9,0)+(-30:.5 and 1)$) arc (90:-90:2);
\filldraw[fill=black!10!white]($(9,0)+(30:.5 and 1)$)arc (90:-90:3) arc (-30:30:.5 and 1) arc (-90:90:2) arc (-30:30:.5 and 1);
\draw[line width=3pt]($(9,0)+(30:.5 and 1)$) arc (30:-30:.5 and 1);
\draw[line width=3pt]($(9,-5)+(30:.5 and 1)$) arc (30:-30:.5 and 1);
\node at (14,-2.5){$D^1\times D^n$};
\draw[very thick]($(9,0)+(30:.5 and 1)$) arc (30:330:.5 and 1) arc (90:-90:2)
arc (30:330:.5 and 1) arc (-90:90:3);
\node at (9,1.5){$M_1$};
\node at (9,-3.5){$M_2$};
\end{tikzpicture}

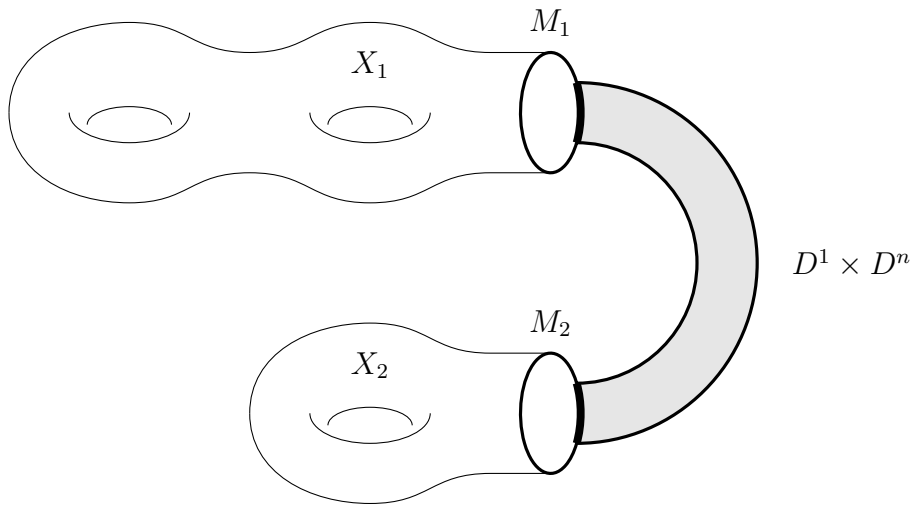
\captionof{figure}{The boundary connected sum $X_1\#_{\p} X_2$}\label{fig:b_conn_sum}
\end{center}

\medskip

The shaded region $D^1\times D^n$ is glued with $X_1$ resp.\ $X_2$ along $\p D^1\times D^n=\{\pm 1\}\times D^n$, which is indicated in the picture by the two thick line segments.  The boundary of $X_1\#_{\p} X_2$, which is diffeomorphic to the connected sum $M_1\# M_2$, is indicated by the medium thick lines.

The connected sum $M_1\# M_2$ is obtained from the disjoint union $M:=(M_1\amalg M_2)$ by removing $S^0\times \mathring D^n=\mathring  D^n\amalg \mathring D^n$ from $M$, and gluing the resulting manifold $M\setminus (S^0\times \mathring D^n)$ along its boundary $S^0\times S^{n-1}$ with the manifold $D^1\times S^{n-1}$ (which has the same boundary). This process is called a {\em $0$\nb-surgery}. More generally, removing $S^k\times \mathring D^{n-k}$ from some $n$\nb-manifold $M$ and gluing back in $D^{k+1}\times  S^{n-k-1}$ is referred to as a {\em $k$\nb-surgery} on $M$, or {\em surgery of codimension $n-k$}. In particular, the connected sum $M_1\# M_2$ is obtained by a codimension $n$\nb-surgery from the disjoint union $M_1\amalg M_2$.

Our next goal is to talk about the connected sum $h_1\# h_2$ of positive scalar curvature metrics $h_1$, $h_2$ on closed manifolds of dimension $n\ge 3$.
According to a fundamental result of Gromov-Lawson \cite[Theorem A]{GL1}, if a closed manifold $M$ carries a \pscm, then any manifold $N$ obtained from $M$ by a surgery of codimension $\ge 3$ also carries a \pscm. In particular, if $M_1$ and $M_2$ are manifolds of dimension $n\ge 3$ which admit \pscm s $h_1$, $h_2$, then also their connected sum $M_1\# M_2$ admits a \pscm\ $h$.

It should be emphasized that this is an {\em existence statement}; the construction of $h$ depends not only on the metrics $h_1$ and $h_2$, but on many additional choices made in the construction of $h$. While it is expected that the connected component $[h]\in \pi_0\scr R^+(M_1\# M_2)$ of the metric $h$ in the space $\scr R^+(M_1\# M_2)$ depends only on $[h_1]\in \pi_0\scr R^+(M_1)$ and $[h_2]\in \pi_0\scr R^+(M_2)$, this is not obvious.

For our definition of $h_1\# h_2$ we will use a refinement of the Gromov-Lawson result due to Gajer \cite{Ga}. Let $M_1$, $M_2$ be closed manifolds of dimension $n$, and consider the boundary connected sum
\begin{equation}\label{eq:trace}
W=([-1,0]\times M_1)\#_{\p} ([-1,0]\times M_2),
\end{equation}
see figure \ref{fig:decomp} below. This manifold is also known as {\em trace of a $0$\nb-surgery on $M_1\amalg M_2$}. According to the second theorem in the introduction of \cite{Ga}, a \pscm\ $h_1\amalg h_2$ on $M_1\amalg M_2$ (which amounts to a \pscm\ $h_1$ on $M_1$ and a \pscm\ $h_2$ on $M_2$) extends to a \pscm\ $g$ on $W$, which is the product metric near the boundary.

\begin{defn}\label{def:conn_sum_pscm} Let $M_1$, $M_2$ be closed manifolds of dimension $n\ge 3$,  and let $h_1$, $h_2$ be \pscm s on $M_1$ resp.\ $M_2$. Then we write $h_1\# h_2$ for any \pscm\ on $M_1\# M_2$ which is the restriction of a metric $g\in \scr R^+(W)$ which
\begin{itemize}
\item restricts to $h_i$ on $M_i\subset \p W$ for $i=1,2$, and
\item is a product metric near the boundary.
\end{itemize}
\end{defn}

We do not claim that the connected component $[h_1\# h_2]\in \pi_0\scr R^+(M_1\# M_2)$ depends only on the connected components of $h_1$ and $h_2$. Fortunately, the ambiguity of the \pscm\ $h_1\# h_2$ does not affect the calculation of our index invariant as the following result shows.

\begin{prop}\label{prop:add}
Let $X_1$, $X_2$ be compact spin manifolds of dimension $4k\ge 4$ with non-empty boundary $\p X_i=M_i$, and let $h_i$ be a \pscm\ on $M_i$.  Then
\[
I(X_1\#_{\p} X_2,h_1\# h_2)=I(X_1,h_1)+I(X_2,h_2).
\]
\end{prop}

We note that if $X_2=D^{4k}$, then $X_1\#_{\p} X_2$ is diffeomorphic to $X_1$. Hence $h_1\# h_2$ can be interpreted as a \pscm\ on $X_1$, and the above formula gives
\[
I(X_1,h_1\# h_2)=I(X_1,h_1)+I(D^{4k},h_2),
\]
as claimed in Proposition \ref{prop:pscm_conn_sum}.

\begin{proof}[Proof of Proposition \ref{prop:add}]
The proof is obtained by applying Gluing Lemma \ref{lem:gluing} to calculate the invariant $I(X_1\#_{\p} X_2,h_1\# h_2)$ in terms of a decomposition of $X_1\#_{\p} X_2$ into three pieces as follows. Identifying a tubular neighborhood of $M_i=\p X_i\subset X_i$ with $[-1,0]\times M_i$, we denote by $X_i'$ the manifold with boundary $X_i':=X_i\setminus (-1,0]\times M_i$, and by $M_i'$ its boundary $M_i':=\p X_i'=\{-1\}\times M_i$. Then $M_1'$ and $M_2'$ are hypersurfaces in $X_1\#_{\p} X_2$ which decompose this manifold into three manifolds with boundary: $X_1'$, $X_2'$ and the manifold $W$ discussed in \eqref{eq:trace}, the boundary connected sum $W$ of $[-1,0]\times M_1$ and $[-1,0]\times M_2$. Here is a picture of this decomposition.

\begin{center}
\begin{tikzpicture}[scale=.8]
\draw (0,0).. controls (0,1) and (1,1.5) ..
(2,1.5) .. controls (3,1.5) and (3,1) ..
(4,1).. controls (5,1) and (5,1.5) ..
(6,1.5) .. controls (7,1.5) and (7,1) ..
(8,1)--(9,1);
\draw (0,0).. controls (0,-1) and (1,-1.5) ..
(2,-1.5) .. controls (3,-1.5) and (3,-1) ..
(4,-1).. controls (5,-1) and (5,-1.5) ..
(6,-1.5) .. controls (7,-1.5) and (7,-1) ..
(8,-1)--(9,-1);
\draw (1,0) arc (180:360:1 and .5);
\draw (1.3,-.2) arc (180:0:.7 and .3);
\draw (5,0) arc (180:360:1 and .5);
\draw (5.3,-.2) arc (180:0:.7 and .3);
\draw (9,0) circle (.5 and 1);
\draw[very thick] (8,1) arc (90:270:.5 and 1);
\node at (6,.8){$X'_1$};
\begin{scope}[shift={(4,-5)}]
\draw (0,0).. controls (0,1) and (1,1.5) ..
(2,1.5) .. controls (3,1.5) and (3,1) ..
(4,1)--(5,1);
\draw (0,0).. controls (0,-1) and (1,-1.5) ..
(2,-1.5) .. controls (3,-1.5) and (3,-1) ..
(4,-1)--(5,-1);
\draw (1,0) arc (180:360:1 and .5);
\draw (1.3,-.2) arc (180:0:.7 and .3);
\draw (5,0) circle (.5 and 1);
\draw[very thick] (4,1) arc (90:270:.5 and 1);
\node at (2,.8){$X'_2$};
\end{scope}
\draw ($(9,0)+(30:.5 and 1)$) arc (90:-90:3);
\draw ($(9,0)+(-30:.5 and 1)$) arc (90:-90:2);
\node at (14,-2.5){$D^1\times D^n$};
\draw[very thick]($(9,0)+(30:.5 and 1)$) arc (30:330:.5 and 1) arc (90:-90:2)
arc (30:330:.5 and 1) arc (-90:90:3);
\draw[snake=brace,mirror snake,thick] (7.75,-6.5) --node[below]{$W$}(12.5,-6.5);
\node at (8,1.5){$M'_1$};
\node at (8,-3.5){$M'_2$};
\node at (15,1)(1){$[-1,0]\times M_1$};
\draw[->](1)--(8.2,0);
\node at (15,-6)(2){$[-1,0]\times M_2$};
\draw[->](2)--(8.2,-5);
\end{tikzpicture}

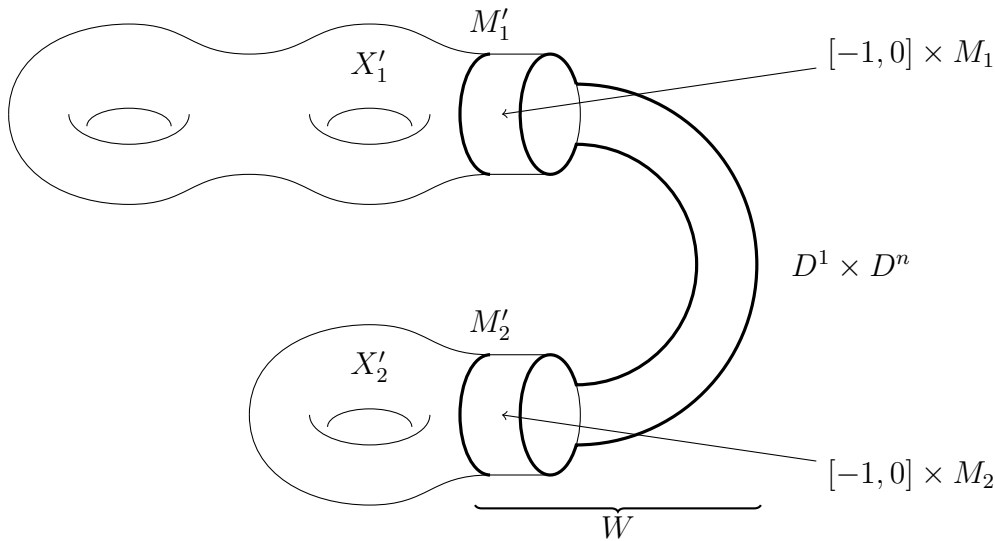
\captionof{figure}{Decomposition of $X_1\#_{\p} X_2$}\label{fig:decomp}
\end{center}

By the result of Gajer \cite{Ga} mentioned above, the \pscm\ $h_1\amalg h_2$ on $M_1\amalg M_2\cong M'_1\amalg M'_2\subset \p W$ extends to a \pscm\ $g$ on $W$ which is a product metric near the boundary. The restriction of $g$ to $M_1\# M_2\subset\p W$ is the \pscm\ we denote by $h_1\# h_2$.

Applying Gluing Lemma \ref{lem:gluing} to the decomposed manifold
\[
X_1\#_{\p} X_2=
(X_1'\amalg X_2')\cup_{M_1'\amalg M_2'} W
\]
we obtain
\begin{align*}
I(X_1\#_{\p} X_2,h_1\# h_2)
=&I(X_1'\amalg X_2',h_1\amalg h_2)+I(W,h_1\amalg h_2\amalg (h_1\# h_2))\\
=&I(X'_1,h_1)+I(X'_2,h_2)+I(W,h_1\amalg h_2\amalg (h_1\# h_2)).
\end{align*}
Here we abuse notation and write $h_i$ for the \pscm\ on $M_i'$ corresponding to $M_i$ via obvious diffeomorphism between $M_i'=\{-1\}\times M_i$ and $M_i$. This diffeomorphism extends to a diffeomorphism between $X_i'$ and $X_i$ and hence $I(X_i',h_i')=I(X_i,h_i)$. Moreover,
\[
I(W,h_1\amalg h_2\amalg (h_1\# h_2))=0
\]
by Lemma \ref{lem:vanishing}, since by construction $g$ is a \pscm\ on $W$ with totally geodesic boundary extending the metric $h_1\amalg h_2\amalg (h_1\# h_2)$ on $\p W$. This proves Proposition \ref{prop:add}.
\end{proof}

\section{Positive scalar curvature metrics on $S^{4k-1}$ with non-trivial $I$-invariant}\label{sec:exotic_pscm}

The purpose of this section is to provide proofs for Theorem \ref{ThmB}, Proposition \ref{prop:intro} and  Proposition \ref{prop:exotic_pscm}. The first step is a characterization of the possible values of $\wh A(Y)\in\Z$ for closed spin manifolds $Y$.

\begin{lem}\label{lem:Ahat} Let $k\ge 1$ be an integer, and let $\ell$ be an integer which is assumed to be even if $k$ is odd.
Then there is a closed spin manifold $Y$ of dimension $4k$ with $\wh A(Y)=\ell$.
\end{lem}

\begin{proof}
There are explicit  spin manifolds of dimension $4$ (resp.\ $8$) with $\wh A$\nb-genus  $2$ (resp.\ $1$):
\begin{itemize}
\item The {\em Kummer surface} $K$ is the complex codimension $1$ submanifold of the complex projective space $\C\P^3$ given by a quartic equation:
\begin{equation}\label{eq:Kummer}
K=\{[z_0,z_1,z_2,z_3]\in \C\P^3\mid z_0^4+z_1^4+z_2^4+z_3^4=0\}.
\end{equation}
This is a closed simply connected spin manifold of dimension $4$ with $\wh A(K)=2$.
\item Joyce has constructed explicit examples of closed simply connected spin $8$\nb-manifolds $J$ with $\wh A(J)=1$ \cite[Section 3]{Joy}. His construction is given by resolving the singularities of the quotient $T^8/\Gamma$ of a finite group $\Gamma$ acting on the $8$\nb-dimensional torus $T^8$. He goes on to show that some of these manifolds admit Riemannian metrics with holonomy $Spin(7)$ \cite[Proposition 2.1.1]{Joy}.
\end{itemize}
Using the multiplicative property
\[
\wh A(Y_1\times Y_2)=\wh A(Y_1)\cdot\wh A(Y_2),
\]
it follows that the Cartesian product $J\times\dots\times J$ (resp.\ $K\times J\times\dots\times J$) provides an example of a $4k$\nb-manifold $Y$ with $\wh A(Y)=1$ for even $k$ (resp.\ $\wh A(Y)=2$ for odd $k$). Then the disjoint union of $m\ge 0$ copies of $Y$ is a spin manifold with $\wh A$\nb-genus$= m$ (for $k$ even), and $\wh A$\nb-genus$= 2m$ (for $k$ odd). Moreover, if $\bar Y$ is the manifold $Y$ with the opposite orientation, then $\wh A(\bar Y)=-\wh A(Y)$ implies that the disjoint union of $m\ge 0$ copies of $\bar Y$ has $\wh A$\nb-genus$= -m$ (if $k$ is even), and $\wh A$\nb-genus$= -2m$ (if $k$ is odd).
\end{proof}

As a corollary of this result, we obtain the following statement which implies both Theorem \ref{ThmB} and Proposition \ref{prop:intro}.

\begin{cor}\label{cor:Ahat} Let $X$ be a compact spin manifold of dimension $4k\ge 4$, let $h\in \scr Y^+(\p X)$, and let  $\ell$ be an integer which is assumed to be even if $k$ is odd. Then there is some closed spin manifold $Y$ such that $I(X\# Y,h)=\ell$.
\end{cor}

\begin{proof}
According to Proposition \ref{prop:conn_sum}, $I(X\# Y,h)=I(X,h)+\wh A(Y)$. By Lemma \ref{lem:Ahat} for even $k$, a suitable choice of $Y$ produces any desired integral value for $\wh A(Y)$ and hence for $I(X\# Y,h)$.

For odd $k$, the values for $\wh A(Y)$ are the even integers. So it suffices to show that the index invariant $I(X,h)$ is even for odd $k$, since then $I(X,h)+\wh A(Y)$ gives any desired even integer for a suitable choice of $Y$.

We recall that $I(X,h)$ is the index of the Dirac operator $D^+(X,g)$, where $g$ is a metric on $X$ which extends $h$ and for which $\p X$ is totally geodesic. If $k$ is odd, the spinor bundle of $X$ is a quaternionic bundle, and the operator $D^+(X,g)$ is quaternionic linear. In particular, its kernel and cokernel are quaternionic vector spaces, and hence they are even dimensional as complex vector spaces. In particular, the index of $D^+(X,g)$ is even.
\end{proof}

The rest of this section is devoted to a proof of Proposition \ref{prop:exotic_pscm} for $k\ge 2$.
In other words, we construct \pscm s $h$ on the sphere $S^{4k-1}$ for $k\ge 2$ with non-trivial index invariant $I(D^{4k},h)\in \Z$. More precisely, we need to show that there are \pscm s $h_\ell\in \scr R^+(S^{4k-1})$ such that $I(D^{4k},h_\ell)=\ell$ for any integer $\ell$ if $k$ is even, and for any even integer $\ell$ if  $k$ is odd.

Our strategy for constructing \pscm s $h$ with $I(D^{4k},h)\ne 0$ is based on the following result.

\begin{lem}\label{lem:Index} Let $Y$ be a closed spin manifold of dimension $4k$ and let $g$ be a \pscm\ on $X:=Y\setminus \mathring D^{4k}$ which is a product metric near the boundary. Let $h\in \scr R^+(S^{4k-1})$ be the restriction of $g$ to $\p X=S^{4k-1}$. Then $I(D^{4k},h)=\wh A(Y)$.
\end{lem}

\begin{proof} Decomposing $Y$ as $Y=X\cup_{S^{4k-1}}D^{4k}$, Gluing Lemma \ref{lem:gluing} implies
\[
\wh A(Y)=I(Y)=I(X,h)+I(D^{4k},h).
\]
Moreover, since $h$ is the restriction of the \pscm\ $g$ on $X$, which is a product metric near the boundary, Vanishing Lemma  \ref{lem:vanishing} implies $I(X,h)=0$.
\end{proof}

Next we need to address whether for a given closed spin manifold $Y$ of dimension $n+1$, we can construct a \pscm\ $g$ on $X=Y\setminus \mathring D^{n+1}$ which is a product metric near the boundary. The following result shows that this is possible under mild dimension and connectivity restrictions on $Y$.

\begin{prop}\label{prop:pscm}
Let $Y$ be a closed manifold of dimension $n+1\ge 6$ which is $2$\nb-connected, that is, $\pi_i(Y)=0$ for $i\le 2$.  Then $X:=Y\setminus \mathring D^{n+1}$ admits a \pscm\ $g$ which is the product metric near the boundary.
\end{prop}

This proposition is essentially a special case of a result of \cite{Ga}. The main result of that paper, the second theorem of the introduction, is the statement that a \pscm\ on a compact manifold $X$ which is a product metric near $\p X$ can be extended over handles of codimension $\ge 3$. In particular, if  $X$ has a handle decomposition involving only handles of codimension $\ge 3$, then $X$ has a \pscm\ which is a product metric near $\p X$. Gajer then uses techniques developed by Smale in his proof of the h-Cobordism Theorem to deduce the following result, formulated as corollary in the introduction of \cite{Ga}.

\begin{cor}[Gajer] If $X$ is a compact manifold of dimension $n+1\ge 6$ with a connected non-empty boundary $\p X$ and $\pi_1(X,\p X)=\pi_2(X,\p X)=0$, then there exists a \pscm\ on $X$ which is a product metric near the boundary.
\end{cor}

\begin{proof}[Proof of Proposition \ref{prop:pscm}]
The inclusion map $X\into Y$ induces an isomorphism on $\pi_i$ for $i\le 2$. Hence the long exact sequence of homotopy groups of the pair $(X,\p X)=(X,S^{n})$ implies that the hypotheses of Gajer's Corollary are satisfied.
\end{proof}

To prove Proposition \ref{prop:exotic_pscm} for $k\ge 2$, it remains to argue that the closed manifolds of Lemma \ref{lem:Ahat} can be chosen to be $2$\nb-connected for $k\ge 2$. The following lemma shows that  we can modify a given spin manifold of dimension $\ge 6$ by surgeries to make it $2$\nb-connected. Since surgeries do not change the $\wh A$\nb-genus, this finishes the proof of Proposition \ref{prop:exotic_pscm}.

\begin{lem} Let $Y$ be a closed spin manifold of dimension $\ge 6$. Then $Y$ can be made $2$\nb-connected by a sequence of $i$\nb-surgeries for $i=0,1,2$.
\end{lem}

\begin{proof}
If $Y$ is not connected, forming the connected sum of two of its components is a $0$\nb-surgery on $Y$. It reduces the number of connected components of $Y$, and so a finite sequence of $0$\nb-surgeries results in a connected manifold.

If $Y$ is connected, but not simply connected, let $S^1\into Y$ be an embedded circle that represents a non-trivial element of $\pi_1(Y)$. The normal bundle of this circle is trivial, since $Y$ is orientable, and hence we can do a $1$\nb-surgery on that circle. The resulting manifold is again  connected; its fundamental group is the quotient of $\pi_1(Y)$ obtained by modding out by the normal subgroup generated by the embedded circle. Since the fundamental group of the compact space $\pi_1(Y)$ is finitely generated, it follows that a finite sequence of $1$\nb-surgeries results in a simply connected manifold.

Similarly, if $Y$ is simply connected, but not $2$\nb-connected, a non-zero element $\alpha\in \pi_2(Y)$ can be represented by an embedded $2$\nb-sphere. Its normal bundle is trivial thanks to the spin condition, allowing us to do surgery on that $2$\nb-sphere. The dimension assumption $\dim Y\ge 6$ guarantees that $\pi_2$ of the resulting manifold is isomorphic to a quotient of $\pi_2(Y)$, given by modding out by the subgroup generated by the element $\alpha$. Hence a finite sequence of $2$\nb-surgeries results in a $2$\nb-connected manifold.
\end{proof}

%
%
%

\end{document}